\documentclass[11pt]{article}

\usepackage{amsmath}
\usepackage{amssymb}
\usepackage{amsthm}
\usepackage{latexsym}
\usepackage{color}
\usepackage{graphicx}
\usepackage{color}
\usepackage{mathrsfs}

\DeclareSymbolFont{calletters}{OMS}{cmsy}{m}{n}
\DeclareSymbolFontAlphabet{\mathcal}{calletters}

\newtheorem{Theorem}{Theorem}[part]
\newtheorem{Definition}{Definition}[part]

\newtheorem{Lemma}{Lemma}[part]

\newtheorem{Remark}{Remark}[part]

\newcommand{\nc}{\newcommand}
\nc{\esssup}{\mathop{\mathrm{ess\,sup}}}
\nc{\essinf}{\mathop{\mathrm{ess\,inf}}}
\nc{\argmax}{\mathop{\mathrm{arg\,max}}}
\nc{\dint}{\displaystyle\int}

\def \R{\mathbb{R}}

\def \E{\mathbb{E}}

\def \1{\mathds{1}}

\def \Ec{{\cal E}}

\def \l({{\left (}}
\def \r){{\right )}}
\def \l[{{\left [}}
\def \r]({{\right ]}}

\newcommand{\MBFigure}[6]{
$\left. \right.$ \\
\refstepcounter{figure}
\addcontentsline{lof}{figure}{\numberline{\thefigure}{\ignorespaces #5}}
\begin{center}
\begin{minipage}{#1cm}
\centerline{\includegraphics[width=#2cm,angle=#3]{#4}}
\begin{center}
\upshape{F\textsc{ig} \normal
\end{center}
size{\thefigure}. $-$} #5
\end{center}
\label{#6}
\end{minipage}
\end{center}
$\left. \right.$ \\}
\title{Burkholder-Davis-Gundy Inequality  for g- martingale }
\author{Wahid {\sc Faidi} \thanks{University Tunis El Manar} \thanks{Shaqra University } 
\\ e-mail:  \textcolor[rgb]{0.00,0.07,1.00}{faidiwahid@su.edu.sa}
}
\begin{document}
\maketitle
\textbf{Classification:} 60E15
\begin{abstract}
In this work we establish an BDG inequality type for certain nonlinear martingale arizing from BSDE. 
\end{abstract}
\section{Introduction}
The nonlinear expectation, as its name indicates, is a nonlinear generalization of the classical expectation.  It has properties in common with the latter but it differs from it especially in the property of linearity. This operator, which was introduced by Peng \cite{PPeng90}, is widely used in financial mathematics, especially in problems related to model uncertainty, such as risk assessment problems under conditions of knight uncertainty.  An important class of nonlinear expectations is that generated by the BSDE so-called $g-$expectations.
As in the case of classical expectation, a theory of nonlinear martingales has developed over the past two decades. Some generalizations of the results concerning classical martingales have been made for nonlinear martingales.
This paper is organized as follows: Section 2 provides the preliminaries, the necessary
notations, conceptions and some properties about the g-martingale.
In section 3, we further study the main problem of this paper, namely the Burkholder-Davis-Gundy (BDG in shirt) Inequality  for $g-$ martingale when $g$ is generalized Lipchitz generator. The case of a quadratic generator is studied in section 4
\section{g-Martingale  }
Let $T$ be a finite or infinite number in $\mathbb{R}_{+}$, and let $\left(B_t\right)_{t \geq 0}$ be a standard $d$-dimensional Brownian  motion defined on a completed probability space endowed with a filtration $\left(\mathscr{F}_t\right)_{t \geq 0}$ generated by this Brownian motion:
$$
\mathscr{F}_t \triangleq \sigma\left\{B_s: 0 \leq s \leq t\right\} \vee \mathcal{N}, \quad \mathscr{F}_{\infty}=\bigvee_{t \geq 0} \mathscr{F}_t
$$
where $\mathcal{N}$ is the set of all $\mathbb{P}$-null subsets.
For simplification, let $L^2\left(\Omega, \mathscr{F}_t, \mathbb{P}\right)$ be the space of all the $\mathscr{F}_t$ measurable square integral $\mathbb{R}$-valued random variables, and define the adapted process spaces as follows:
\begin{equation*}
\begin{aligned}
& \mathscr{S}^2(0, T ; \mathbb{R}):=\left\{\left(Y_t\right)_{t \in[0, T]}: \begin{array}{l}
Y \text { is the RCLL } \mathbb{R} \text {-valued process, } \\
\text { such that } \mathbb{E}\left[\sup _{0 \leq t \leq T}\left|Y_t\right|^2\right]<+\infty
\end{array}\right\} \\
& \mathscr{H}^2\left(0, T ; \mathbb{R}^d\right):=\left\{\left(Z_t\right)_{t \in[0, T]}: \begin{array}{l}
Z \text { is the adapted } \mathbb{R}^d \text {-valued process, } \\
\text { with } \mathbb{E} \int_0^T\left|Z_t\right|^2 \mathrm{~d} t<+\infty
\end{array}\right\} 
\end{aligned}
\end{equation*}

The generator $g(t, \omega, y, z):[0, T] \times \Omega \times \mathbb{R} \times \mathbb{R}^d \longmapsto \R$ is a random function which is a progressively measurable stochastic process for any $(y, z)$. We assumed that it satisfies the following  assumptions 
\begin{itemize}
\item[\textbf{(H1)}] $\left|g(t, y, z)-g\left(t, y^{\prime}, z^{\prime}\right)\right| \leq u(t)\left|y-y^{\prime}\right|+v(t)\left|z-z^{\prime}\right|$, where $u$ and $v$ are two positive functions mapping from $[0, T]$ to $\mathbb{R}_{+}$, such that $\int_0^T\left[u(t)+v^2(t)\right] \mathrm{d} t<+\infty$;
\item[\textbf{(H2)}] $g(t, y, 0)=0$, for each $y \in \mathbb{R}$, $\mathrm{d} \mathbb{P} \times \mathrm{d} t$ - a.e.
\end{itemize}
The assumption \textbf{(H1)} is a generalized Lipschitz condition, whose Lipschitz constant is replaced by two deterministic functions depending on $t$. Note that under assumptions \textbf{(H1)}-\textbf{(H2)} we have forall $(y,z)\in \mathbb{R}\times \mathbb{R}^d$,
\begin{equation*}
\mathbb{E}\left[\left(\int_0^T|g(t, y,z)| dt\right)^2\right]=\mathbb{E}\left[\left(\int_0^T|g(t, y,z)-g(t,y,0)| dt\right)^2\right] \leq \mathbb{E}\left[\left(\int_0^Tv^2(t)|z|^2 dt\right)^2\right]<+\infty
\end{equation*}
and so under assumptions \textbf{(H1)}-\textbf{(H2)}, according \cite{Chen00},  the BSDE 
\begin{equation}\label{bsde}
Y_t =\xi+\dint_t^{T}g\left(s, Y_s, Z_s\right) d s-\dint_t^{T}Z_s d B_s 
\end{equation}
admits a unique solution $(Y^{\xi},Z^{\xi}) \in \mathscr{S}^2(0, T ; \mathbb{R}) \times \mathscr{H}^2\left(0, T ; \mathbb{R}^d\right)$ for all $\xi \in  L^{2}\left(\Omega, \mathscr{F}_T, P\right)$.
\\ The operator
\begin{equation*}
\begin{split}
\Ec_g:L^{2}\left(\Omega, \mathscr{F}_T, \mathbb{P}\right)&\longmapsto \mathbb{R}
\\& \xi \longmapsto Y^{\xi}_0
\end{split}
\end{equation*}
 is a typical example of nonlinear expectation called g expectation.
 The notion of nonlinear expectation was firstly introduced  by Peng \cite{PPeng90}. It is an operator verifying a certain properties, namely
 \begin{itemize}
 \item[(i)] Strict monotonicity:
 \begin{itemize}
 \item[$\bullet$] 
  if $X_1 \geq X_2, \mathbb{P}-$ a.s., then $\mathcal{E}\left[X_1\right] \geq \mathcal{E}\left[X_2\right]$, and furthermore
 \item[$\bullet$]
  if $X_1 \geq X_2, \mathbb{P}-$ a.s., then $\mathcal{E}\left[X_1\right]=\mathcal{E}\left[X_2\right] \Longleftrightarrow X_1=X_2, \mathbb{P}-$ a.s.
\end{itemize}
\item[(ii)] preserving of constants: $\mathcal{E}[c]=c$, for each constant $c$.

 \end{itemize}

\begin{Definition}{Conditional $g$ -expectation}
\\
 The conditional $g$ -expectation of $\xi$ with respect to $\mathscr{F}_t$ is defined by
$$
\mathcal{E}_{g}\left[\xi \mid \mathscr{F}_t\right]=Y^{\xi}_t
$$
Where $(Y^{\xi}, Z^{\xi})$ is the unique solution of the BSDE (\ref{bsde}).
\\
If $\tau \leq T$ is a stopping time, we define similarly
$$
\mathcal{E}_{g}\left[X \mid \mathscr{F}_{\tau}\right]=Y^{\xi}_\tau
$$
\end{Definition}
\begin{Definition}
  $A$ process $\left(Y_{t}\right)_{0 \leq t \leq T}$ such that $E\left[Y_{t}^{2}\right]<\infty$ for all $t$ is a g-martingale (resp. $g$ -supermartingale, $g$ -submartingale) iff
$$
\mathcal{E}_{g}\left[Y_{t} \mid \mathscr{F}_{s}\right]=Y_{s}, \quad\left(\text { resp. } \leq Y_{s}, \geq Y_{s}\right), \quad \forall s \leq t \leq T
$$
\end{Definition}
\section{BDG inequality for g-martingale}
\begin{Remark}
Condition (ii) and (iii) implies 
$$\forall (t,y,z) \in \R \times \R \times  \R^{d} ; 
\left|g\left(t, y, z\right)\right|\leq  v(t)\left|z\right|$$
indeed
\begin{align*}
\left|g\left(t, y, z\right)\right|&=\left|g\left(t, y, z\right)-g\left(t, y, 0\right)\right|
\\& 
u(t)\left|y-y\right|+v(t)\left|z-0\right|= v(t)\left|z\right|
\end{align*}
\end{Remark}

\begin{Lemma}(Lenglart)
 Let $\left(X_{t}\right)_{t \geq 0}$ be a positive adapted right-continuous process dominated
by a predictable increasing process  $\left(A_{t}\right)_{t>0}$ i.e   for every bounded stopping time $\tau$, $\mathbb{E}\left(X_{\tau} \right) \leq \mathbb{E}\left(A_{\tau} \right) .$ Then, for every $k \in(0,1)$,
$$ \mathbb{E}\left(\left(\sup _{ t \geq 0} X_{t}\right)^{k}\right) \leq \frac{2-k}{1-k} \mathbb{E}\left(A_{\infty}^{k}\right)$$
\end{Lemma}
\begin{Theorem}
For any $1 \leq p <+\infty$, there exists two constants $c^g_p$ and $C^g_p$ such that for all  $g$-martingale $Y$ vanishing at zero;$$ c^g_p \E[\langle Y \rangle^{\frac{p}{2}}_T] \leq \E[ (Y^*)^p_\infty ]\leq C^g_p \E[\langle Y \rangle^{\frac{p}{2}}_T]$$
\end{Theorem}
\begin{proof}

We start by proving the left hand side inequality. 
\\
For each integer $n \geqslant 1$, let us introduce the stopping time
$$
\tau_{n}=\inf \left\{t \in[0, T], \int_{0}^{t}\left|Z_{r}\right|^{2} \mathrm{~d} r \geqslant n\right\} \wedge T
$$
Itô's formula gives us
$$
\int_{0}^{\tau_{n}}\left|Z_{s}\right|^{2} ds=\left|Y_{\tau_{n}}\right|^{2}+ \int_{0}^{\tau_{n}} 2Y_{s} g\left(s, Y_{s}, Z_{s}\right)ds -2 \int_{0}^{\tau_{n}} Y_{s} Z_{s} \mathrm{~d} B_{s}
$$
But, from the assumption on $g,$ we have  $g\left(s, y, z\right) \leqslant v(s)|z|$, and so

$$
2| yg(s, y, z)| \leqslant 2v^2(s) |y|^{2}+\dfrac{1}{2}|z|^{2} 
$$
Thus, since $\tau_{n} \leqslant T,$ we deduce that
$$
\frac{1}{2} \int_{0}^{\tau_{n}}\left|Z_{s}\right|^{2} \mathrm{~d} s \leqslant Y_{*}^{2}+2 \mu Y_{*}^{2} +2\left|\int_{0}^{\tau_{n}} Y_{s} Z_{s} \mathrm{~d} B_{s}\right| .
$$
Where $\mu:=\dint_0^T v^2(s)ds$. It follows that
$$
\int_{0}^{\tau_{n}}\left|Z_{s}\right|^{2} \mathrm{~d} s \leqslant (2+4\mu )Y_{*}^{2} +4\left|\int_{0}^{\tau_{n}} Y_{s} Z_{s} \mathrm{~d} B_{s}\right| .
$$
and thus that
\begin{equation*}
\left(\int_{0}^{\tau_{n}}\left|Z_{s}\right|^{2} \mathrm{~d} s\right)^{p / 2} \leqslant k_{p}\left(Y_{*}^{p}+\left|\int_{0}^{\tau_{n}} Y_{s} Z_{s} dB_{s}\right|^{p / 2} \right)
\end{equation*}
Hence
\begin{equation}\label{ineq1}
\mathbb{E}\left[\left(\int_{0}^{\tau_{n}}\left|Z_{s}\right|^{2} \mathrm{~d} s\right)^{p / 2}\right] \leqslant k_{p}\left(\mathbb{E}\left[Y_{*}^{p}\right]+\mathbb{E}\left[\left|\int_{0}^{\tau_{n}} Y_{s} Z_{s} dB_{s}\right|^{p / 2} \right]\right)
\end{equation}
But by the BDG inequality, we get
$$
\begin{aligned}
\mathbb{E}\left[\left|\int_{0}^{\tau_{n}} Y_{s} Z_{s} dB_{s}\right|^{p / 2}\right] 
& \leqslant 
c_{p} \mathbb{E}\left[\left(\int_{0}^{\tau_{n}}\left|Y_{s}\right|^{2}\left|Z_{s}\right|^{2}ds\right)^{p/4}\right] \\
& \leqslant c_{p} \mathbb{E}\left[Y_{*}^{p / 2}\left(\int_{0}^{\tau_{n}}\left|Z_{s}\right|^{2} ds\right)^{p / 4}\right]
\end{aligned}
$$
$$
\mathbb{E}\left[\left|\int_{0}^{\tau_{n}} Y_{s} Z_{s} dB_{s}\right|^{p / 2}\right] \leqslant \frac{c_{p}^{2}}{2} \mathbb{E}\left[Y_{*}^{p}\right]+\frac{1}{2} \mathbb{E}\left[\left(\int_{0}^{\tau_{n}}\left|Z_{s}\right|^{2} ds\right)^{p / 2}\right]
$$
and so , we get, for each $n \geqslant 1$
$$
\mathbb{E}\left[\left|\int_{0}^{\tau_{n}} Y_{s} Z_{s} dB_{s}\right|^{p / 2}\right] \leqslant c_{p}^{2} \mathbb{E}\left[Y_{*}^{p}\right]
$$
Plugging the last inequality in inequality (\ref{ineq1}) we obtain, for each $n \geqslant 1$
$$
\mathbb{E}\left[\left(\int_{0}^{\tau_{n}}\left|Z_{s}\right|^{2} \mathrm{~d} s\right)^{p / 2}\right] \leqslant d_{p} \mathbb{E}\left[Y_{*}^{p}\right]
$$
Fatou's lemma implies that
$$
\mathbb{E}\left[\left(\int_{0}^{T}\left|Z_{s}\right|^{2} ds\right)^{p/2}\right] \leqslant c_{p}^{2} \mathbb{E}\left[Y_{*}^{p}\right]
$$
We proceed now to the proof of the right hand side inequality.
By stopping it is enough to prove the result for bounded $Y$. Let $q \geq 2 .$ From Itô's formula we have
\begin{align*}
d\left|Y_{t}\right|^{q}&
=q\left|Y_{t}\right|^{q-1} \operatorname{sgn}\left(Y_{t}\right) d Y_{t}+\frac{1}{2} q(q-1)\left|Y_{t}\right|^{q-2} d\langle Y\rangle_{t}
\\&
=q \operatorname{sgn}\left(Y_{t}\right)\left|Y_{t}\right|^{q-1} (-g(t, Y_t, Z_t)dt +Z_t d B_{t})+\frac{1}{2} q(q-1)\left|Y_{t}\right|^{q-2} Z_t^2dt
\\&
=-q \operatorname{sgn}(Y_{t})\left|Y_{t}\right|^{q-1} g(t, Y_t, Z_t)dt 
+\frac{1}{2} q(q-1)\left|Y_{t}\right|^{q-2} Z_t^2dt
+q \operatorname{sgn}\left(Y_{t}\right)\left|Y_{t}\right|^{q-1} Z_t d B_{t}
\end{align*}
\begin{align*}
|Y_{t}|^{q}
=\int_0^t-q \operatorname{sgn}(Y_{s})\left|Y_{s}\right|^{q-1} g(t, Y_s, Z_s)dt 
+\frac{1}{2} q(q-1)\left|Y_{s}\right|^{q-2} Z_s^2ds
+\int_0^t q \operatorname{sgn}\left(Y_{s}\right)\left|Y_{s}\right|^{q-1} Z_s d B_{s}
\end{align*}
\begin{align*}
\mathbb{E}\left(\left|Y_{t}\right|^{q} \mid \mathcal{F}_{0}\right)& \leq
\mathbb{E}\left(\int_0^t-q \operatorname{sgn}(Y_{s})\left|Y_{s}\right|^{q-1} g(t, Y_s, Z_s)ds 
+\frac{1}{2} q(q-1)\left|Y_{s}\right|^{q-2} Z_s^2ds\mid \mathcal{F}_{0}\right)
\\&
\mathbb{E}\left(\int_0^t q v(s) \left|Y_{s}\right|^{q-1} |Z_s|ds 
+\frac{1}{2} q(q-1)\left|Y_{s}\right|^{q-2} Z_s^2ds\mid \mathcal{F}_{0}\right)
\end{align*}
From the Lenglart's domination inequality, we deduce then that for every $k \in(0,1)$,
\begin{align*}
\mathbb{E}\left(\left(\sup _{0 \leq t \leq T}\left|Y_{t}\right|^{q}\right)^{k}\right) &\leq \frac{2-k}{1-k}\mathbb{E}\left(\int_0^T q v(s) \left|Y_{s}\right|^{q-1} |Z_s|ds 
+\frac{1}{2} q(q-1)\left|Y_{s}\right|^{q-2} Z_s^2ds\right)^k
\\
& \leq 
\frac{2-k}{1-k}\mathbb{E}\left(\int_0^T q  \left|Y_{s}\right|^{q-2} (\dfrac{\delta^2}{2}|Z_s|^2+\frac{v^2(s)}{2\delta^2}|Y_s|^2)+\frac{1}{2} q(q-1)\left|Y_{s}\right|^{q-2} Z_s^2ds\right)^k
\\
&=
\frac{2-k}{1-k}\mathbb{E}\left(\int_0^T \frac{q v^2(s)}{2\delta^2} \left|Y_{s}\right|^{q}  
+(\frac{1}{2} q(q-1+ \delta^2)) \left|Y_{s}\right|^{q-2} Z_s^2ds\right)^k
\\
& \leq 
\frac{2-k}{1-k}\mathbb{E}\left(\int_0^T \frac{q v^2(s)}{2\delta^2} \left|Y_{s}\right|^{q}ds\right)^k  
+\frac{2-k}{1-k}\mathbb{E}\left(\int_0^T (\frac{1}{2} q(q-1+ \delta^2)) \left|Y_{s}\right|^{q-2} Z_s^2ds\right)^k
\\& \leq 
\frac{2-k}{1-k}(\frac{q \mu}{2\delta^2})^{k}\mathbb{E}\left(\left(\sup _{0 \leq t \leq T}\left|Y_{t}\right|^{q}\right)^{k}\right)
\\&+\frac{2-k}{1-k}(\frac{1}{2} q(q-1+\delta^2)^k\mathbb{E}\left(\left(\sup _{0 \leq t \leq T}\left|Y_{t}\right|^{k(q-2)}\right)\left(\int_0^T   Z_s^2ds\right)^k\right)
\end{align*}
Where $\delta$  is a strictly positive constant. Therefore
\begin{align*}
(1-\frac{2-k}{1-k}(\frac{q \mu}{2\delta^2})^{k})\mathbb{E}\left(\left(\sup _{0 \leq t \leq T}\left|Y_{t}\right|^{q}\right)^{k}\right) &\leq 
\frac{2-k}{1-k}(\frac{1}{2} q(q-1+ \delta^2)^k\mathbb{E}\left(\left(\sup _{0 \leq t \leq T}\left|Y_{t}\right|^{k(q-2)}\right)\left(\int_0^T   Z_s^2ds\right)^k\right)
\end{align*}
By Holder inequality we obtain
\begin{align*}
(1-\frac{2-k}{1-k}(\frac{q \mu}{2\delta^2})^{k})\mathbb{E}\left((\sup _{0 \leq t \leq T}\left|Y_{t}\right|)^{qk}\right) &\leq 
\frac{2-k}{1-k}(\frac{1}{2} q(q-1+ \delta^2)^k
\left(\mathbb{E}(\sup _{0 \leq t \leq T}\left|Y_{t}\right|)^{kq}\right)^{1-\frac{2}{q}}
\\&
\left(\mathbb{E}\left(\int_0^T Z_s^2ds\right)^\frac{kq}{2}\right)^{\frac{2}{q}}
\end{align*}
By by choosing $\delta$ large enough such that $\kappa= 1-\frac{2-k}{1-k}(\frac{q \mu}{\delta^2}T)^{k}>0$ and taking $p=qk$, we obtain
\begin{align*}
\mathbb{E}\left((\sup _{0 \leq t \leq T}\left|Y_{t}\right|)^{p}\right) &\leq 
\frac{2-k}{\kappa(1-k)}(\frac{1}{2} q(q-1)+q\mu \delta^2)^k
\left(\mathbb{E}\int_0^T Z_s^2ds\right)^\frac{p}{2}
\end{align*}
The result follows.
\end{proof}
\section{Quadratic generator case}
In this paragraph, we show using a counterexample that the previous inequality is no longer valid in the quadtratic case.
\\
For $n \in \mathbb{N}$, let $Y^n$ the stochastic processes defined by 
$$Y^n_t=nB_t-n^2t; 0\leq t \leq T$$
It's clear that for all $n \in \mathbb{N}$; the pair $(Y^n,n)$ is solution of the quadratic BSDE
$$dY_t=-Z_t^2dt+ Z_t dB_t; Y_T=nB_T-n^2T$$
Therefore for all $n \in \mathbb{N}$, $Y^n$ is a $g-$martingale with $g(z)=-z^2$
\\
If the BDG inequality  holds for $Y$ we will have
$$
|\E(Y^n_T)|\leq \E[\sup_{0\leq t \leq T}|Y_t^n|] \leq C(T)\E[\langle Y^n\rangle_T^\frac{1}{2}]
$$
That's means, for all $n \in \mathbb{N}$
$$
n^2T \leq n C(T)\sqrt{T}
$$
Which is absurd.


\begin{thebibliography}{99}

\bibitem{Peng97}
Peng, S. 1997. Backward SDE and related g-expectations. In Backward Stochastic Differential Equations; El Karoui, N., andMazliak, L., Eds. Pitman Research Notes inMathematics Series, 364. London: Longman Scientific \& Technical, 141–159.
\bibitem{PPeng90}
Pardoux, E., Peng, S.: Adapted solution of a backward stochastic differential equation,
Systems and Control Letters 14(1), 55–61 (1990)
\bibitem{lenglart77} E. Lenglart, Relation de domination entre deux processus, Ann. Inst. H. Poincar´e
Sect. B (N.S.) 13 (1977), no. 2, 171–179. MR 0471069 (57 \#10810)
\bibitem{Chen00}
Z. Chen and B. Wang, Infinite time interval BSDEs and the convergence of g-martingales, J. Austral. Math. Soc. Ser. A 69 (2000), no. 2, 187–211.
\bibitem{FJD11}
S. Fan, L. Jiang, and D. Tian, One-dimensional BSDEs with finite and infinite time horizons, Stochastic Process. Appl. 121 (2011), no. 3, 427–440.
\end{thebibliography}
\end{document}